\documentclass[11pt]{article}
\usepackage{amsmath}
\usepackage{amsthm}
\usepackage{amssymb}

\newtheorem{thm}{Theorem}

\begin{document}

\vspace*{30px}

\begin{center}\Large
%%% Insert the title of your talk here %%%
\textbf{A Short Combinatorial Proof of Derangement Identity}
\bigskip\large

%%% Insert your first name and family name here %%%
Ivica Martinjak\\
Faculty of Science, University of Zagreb\\
Bijeni\v cka cesta 32, HR-10000 Zagreb, Croatia\\
and\\
Dajana Stani\'c\\
Department of Mathematics, University of Osijek\\
Trg Ljudevita Gaja 6, HR-31000 Osijek, Croatia\\
%%% Insert your institution here %%%

%%% Insert you city and country here %%%
\end{center}

%%% Give the abstract of your talk here %%%

\begin{abstract}
The $n$-th rencontres number with the parameter $r$ is the number of permutations having exactly $r$ fixed points. In particular, a derangement is a permutation without any fixed point. We presents a short combinatorial proof for a weighted sum derangement identities.
\end{abstract}

\noindent {\bf Keywords:} derangements, rencontres numbers, recurrence relation, factorial, binomial coefficient\\
\noindent {\bf AMS Mathematical Subject Classifications:}  	05A05, 05A10

\section{Introduction}

Having a permutation $\sigma \in S_n$, $\sigma : [n] \to [n]$ where $[n]:=\{1,2,\ldots,n\}$, it is said that $k \in [n]$ is a {\it fixed point} if it is mapped to itself, $\sigma(k)=k$. Permutations without fixed points are of particular interest and are usually called {\it derangements}. We let $D_n$ denote the number of derangements of the set $[n]$, $D_n= |S_n^{(0)}| $, $$S_n^{(0)}:=\{\sigma \in S_n: \sigma(k) \neq k, k=1,\ldots, n\}.$$ 

Derangements are usually introduced in the context of inclusion-exclusion principle \cite{Bona,GKP,Stan}, since this principle is used to provide an interpretation of $D_n$ as a {\it subfactorial},
\begin{eqnarray} \label{themainformula}
D_n= n! \sum_{k=0} ^n \frac{(-1)^k}{k!}.
\end{eqnarray}
%Further properties of the derangement numbers arise from (\ref{themainformula}), where
%\begin{eqnarray}
%D_n&=&  \bigg\lfloor   \frac{n!}{e}  + \frac{1}{2}  \bigg \rfloor
%\end{eqnarray}
%is the best known one.
The numbers $D_0, D_1, D_2,\ldots, D_n, \ldots$ form recursive sequence $(D_n)_{n \ge0}$ defined by the recurrence formulae
\begin{eqnarray} \label{themainrecurrence}
D_n= (n-1)(D_{n-1} + D_{n-2})
\end{eqnarray}
and initial terms $D_0=1, D_1=0$. There is a counting argument to prove this. Let the number $k$ be mapped by $\sigma$ to the number $j$, $j=1,\ldots, k-1, k+1, \ldots, n$. Note that there are $(n-1)$ such permutations $\sigma$. Now, we separate the set of permutations $\sigma$ into two disjoint sets ${\cal A}$ and ${\cal B}$, such that
\begin{eqnarray*}
{\cal A}&:=& \{ \sigma \in S_n^{(0)} : \sigma(j) \neq k, \sigma(k)=j \}\\
{\cal B}&:=& \{ \sigma \in S_n^{(0)} : \sigma(j) = k, \sigma(k) =j \}.
\end{eqnarray*}
This means that
$$D_n=(n-1) (|{\cal A}| + |{\cal B}|).$$
The set ${\cal A}$ counts $D_{n-1}$ elements while the set ${\cal B}$ counts $D_{n-2}$ elements. The fact that the number $k$ in this reasoning is chosen without losing generality, completes the proof of (\ref{themainrecurrence}).

By a simple algebraic manipulation with (\ref{themainrecurrence}) we obtain another recurrence for the sequence $(D_n)_{n \ge 0}$,
\begin{eqnarray} \label{thesecondrecurrence}
D_n&=&n D_{n-1} + (-1)^n.
\end{eqnarray}
Namely, it holds true
$$
nD_{n-2} - D_{n-1} - D_{n-2}= -nD_{n-3}+D_{n-2} + 2D_{n-3} = \cdots = (-1)^n
$$
and this immediately gives the above recurrence from (\ref{themainrecurrence}). 

When we iteratively apply recurrence (\ref{thesecondrecurrence}) to the derangement number on the r.h.s. of  this relation we get
\begin{eqnarray*}
nD_{n-1} +(-1)^n &=& n  \big [  (n-1)D_{n-2} + (-1)^{n-1}    \big] +(-1)^n\\
\end{eqnarray*}
which finally results with
\begin{eqnarray} \label{telescopingresult}
 n(n-1)(n-2) \cdots 3 (-1)^2 + n(n-1)(n-2) \cdots 4(-1)^3 + \cdots +(-1)^n.
\end{eqnarray}
on the r.h.s. of (\ref{thesecondrecurrence}), which completes the proof of (\ref{themainformula}).

A few identities for the sequence $(D_n)_{n \ge0}$ are known \cite{DeEl,Hass,PWZ}. In \cite{DeEl} Deutsche and Elizalde give a nice identity
\begin{eqnarray} \label{DeEl}
D_n&=& \sum_{k=2}^n (k-1) \binom{n}{k} D_{n-k}.
\end{eqnarray}
Recently, Bhatnagar presents families of identities for some sequences including the shifted derangement numbers \cite{Bhat}, deriving it using an Euler's identity \cite{Bhat2}. In what follows we demonstrate a combinatorial proof for that derangement identitiy, with weighted sum.

\section{A pair of weighted sums for derangements}

%We shall now present some identities for derangements as a consequence of recurrence relations (\ref{themainrecurrence}) and (\ref{thesecondrecurrence}).

We define the {\it rencontres number} $D_n(r)$ as the number of permutations $\sigma \in S_n$ having exactly $r$ fixed points. Thus, $D_n(0) = D_n$. For a given $r \in \mathbb{N}$, we define the sequence $D_0(r), D_1(r), \ldots, D_n(r), \ldots$, denoted $(D_n(r))_{n \ge r}$.

Applying an analogue counting argument that we used when proving relation (\ref{themainrecurrence}), one can represents rencontres numbers by the derangement numbers,
\begin{eqnarray} \label{derangementsrencontres}
D_n(r) = \binom{n}{r} D_ {n-r}.
\end{eqnarray}

On the other hand, relation (\ref{derangementsrencontres}) follows immediately from the fact that fixed points here are $r$-combinations over the set of $n$ elements.

A few other notable properties of the rencontres numbers is also known. The difference between numbers in the sequences $(D_n)_{n \ge 0}$ and $(D_n(1))_{n \ge 1}$ alternate for the value 1, which follows from (\ref{thesecondrecurrence}). According to the definition of rencontres numbers, the sum of the $n$-th row in the array of numbers $(D_n(r))_{n \ge r}$ is equal to $n!$,
\begin{eqnarray}
n! = \sum_{k=0}^n D_n(k).
\end{eqnarray}
Moreover, identity (\ref{derangementsrencontres}) shows that $D_n$ can be interpreted as a weighted sum of rencontres numbers in the $n$-th row of the array, by means of relation (\ref{DeEl}),
\begin{eqnarray}
D_n&=& \sum_{k=2}^n (k-1) D_{n}(k).
\end{eqnarray}

The number $D_n/(n-1)$ is also a weighted sum of previous consecutive derangement numbers. For example, $24 + 12 D_2 + 4 D_3 + D_4  = \frac{D_6}{5}$. In general we have
%$60 D_2 + 20 D_3 + 5 D_4 + D_5 + 5! = \frac{D_7}{6}$, \ldots
\begin{eqnarray}\label{Lemma1Prototype}
n! +  \sum_{k=2}^{n} \frac{n!}{k!} D_k = \frac{D_{n+2}}{n+1},
\end{eqnarray}
as follows from Theorem \ref{firsttheorem}.

\begin{thm} \label{firsttheorem}
For $n \in \mathbb{N}$ and the sequence of derangement numbers $(D_n)_{n \ge 0}$ we have
\begin{eqnarray} \label{thefirstdual}
1 + \sum_{k=1}^n \frac{D_k}{k!} = \frac{D_{n+2}}{(n+1)!}.
\end{eqnarray}
\end{thm}

\begin{proof}
Within a derangement $\sigma$, the number $k$, $k=1,\ldots, n$ can be mapped to any $j$, $j=1, \ldots, k-1, k+1, \ldots, n$. We let ${\cal A}_n$ denote the set of derangements with $\sigma(k)=j$, where $j \neq k$, $${\cal A}_n :=\{\sigma \in S_n^{(0)} : \sigma(k)=j\}.$$ Obviously, cardinality of the set ${\cal A}_n$ is invariant to the choice of $j$, $j \neq k$. More precisely,
$$|{\cal A}_n | = \frac{D_n}{n-1}.$$

Furthermore, we separate the set ${\cal A}_n $ into two disjoint sets of derangements, sets ${\cal B}_n $ and ${\cal C}_n $,
\begin{eqnarray*}
{\cal B}_n   &:=&\{\sigma \in {\cal A}_n    : \sigma(j) = k   \} \\
{\cal C}_n   & :=&\{ \sigma \in {\cal A}_n  : \sigma(j) \neq k  \}.
\end{eqnarray*}
Obviously, the set ${\cal B}_n $ counts $D_{n-2}$ elements. For derangements in ${\cal C}_n$ there are now $(n-2)$  equivalent ways to map $j$ (excluding $j$ and $k$), as Figure \ref{Fig1} illustrates. Thus,  we have $$|{\cal C}_n | = (n-2) | {\cal A}_{n-1}|,$$
which gives the recurrence relation
\begin{eqnarray} \label{recproofdoublecount}
|{\cal A}_n|= D_{n-2} + (n-2) |{\cal A }_{n-1}  |.
\end{eqnarray}
After repeating usage of (\ref{recproofdoublecount}) we get identity (\ref{Lemma1Prototype}) which completes the proof.
\end{proof}

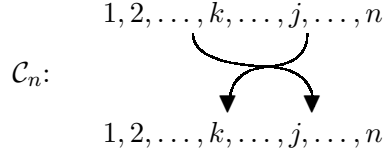
\begin{figure}[h]
\begin{center}  
\setlength{\unitlength}{0.8mm}
\begin{picture}(100, 25)

%\put(30,65){${\cal A}_n$:}
%\put(45,65){$1,2,\ldots,k,\ldots,j,\ldots,n$}
%\put(45,45){$1,2,\ldots,k,\ldots,j,\ldots,n$}
%\qbezier(79,63)(79,58)(72.5,57.5)
%\qbezier(72.5,57.5)(66,57)(66,51)
%\put(64,50){$\blacktriangledown$}

%\put(0,20){${\cal B}_n$:}
%\put(15,20){$1,2,\ldots,k,\ldots,j,\ldots,n$}
%\put(15,0){$1,2,\ldots,k,\ldots,j,\ldots,n$}
%\qbezier(49,18)(49,13)(42.5,12.5)
%\qbezier(42.5,12.5)(36,12)(36,6)
%\put(34,5){$\blacktriangledown$}

%\qbezier(35,18)(35,13)(42.5,12.5)
%\qbezier(42.5,12.5)(50,12)(50,6)
%\put(48,5){$\blacktriangledown$}

\put(0,10){${\cal C}_n$:}
\put(15,20){$1,2,\ldots,k,\ldots,j,\ldots,n$}
\put(15,0){$1,2,\ldots,k,\ldots,j,\ldots,n$}
\qbezier(49,18)(49,13)(42.5,12.5)
\qbezier(42.5,12.5)(36,12)(36,6)
\put(34,5){$\blacktriangledown$}

\qbezier(30,18)(30,13)(42.5,12.5)
\qbezier(42.5,12.5)(50,12)(50,6)
\put(48,5){$\blacktriangledown$}

\end{picture}
\end{center}
\caption{In case of derangements in the set ${\cal C}_n $ there are $(n-2)$  equivalent ways to map $j$.}  \label{Fig1}
\end{figure}

In order to prove Theorem \ref{firsttheorem} algebraically, we apply recurrence (\ref{themainrecurrence}) to get
\begin{eqnarray*}
\frac{D_{n+2}}{(n+1)!} &=& \frac{(n+1)(D_{n+1} +D_n )}{(n+1)!}= \frac{D_{n+1}}{n!} + \frac{D_n}{n!}\\
 &=& \frac{n(D_n+D_{n-1})}{n!}  +\frac{D_n}{n!} = \frac{D_n}{(n-1)!} +\frac{D_{n-1}}{(n-1)!}+ \frac{D_n}{n!} \\
 &=& 1 + \frac{D_1}{1!} + \cdots + \frac{D_n}{n!}  = 1+ \sum_{k=1}^{n} \frac{D_k}{k!}.\\
\end{eqnarray*}
%Note that Lemma \ref{lemma1} is a basic specialization of Bhatnagar's relation (\ref{Bhat1}), in case of the derangement numbers.

%Identities presented in the next theorems, Theorem \ref{thefirsttheorem} and Theorem \ref{thesecondtheorem}, can also be interpreted in the context of the %array $(D_n)_{n \ge 0}$.

%It is worth mentioning that the ratio $P_n:=D_n/n!$ can be interpreted as a probability $P_n$ that a randomly chosen permutation of $S_n$ is a derangement. So, %Theorem \ref{firsttheorem} sums probabilities $P_i$, $i=0,1, \ldots, n$. This sum is bigger then $P_{n+1}$, i.e. it is equal to $P_{n+1} \frac{D_{n+2}}{D_{n+1}}$.

%Similarly as in case of Lema \ref{lemma1}, we shall prove the identity presented in Lemma \ref{lemma2}.

\begin{thm} \label{secondtheorem}
For $n \in \mathbb{N}$ and the sequence of derangement numbers $(D_n)_{n \ge 0}$ we have
\begin{eqnarray} \label{theseconddual}
1+ \sum_{k=1}^n \frac{(-1)^k D_{k+3}}{k+2} = (-1)^n D_{n+2}.
\end{eqnarray}
\end{thm}
\begin{proof}
By applying recurrence (\ref{themainrecurrence}) we have
\begin{eqnarray*}
\sum_{k=0}^n \frac{(-1)^k D_{k+3}}{k+2} &=& \frac{2(D_2+D_1)}{2} - \frac{3(D_3+D_2)}{3} + \cdots (-1)^n \frac{(n+2)(D_{n+2} + D_{n+1})}{n+2}\\
&=& (D_2 + D_1) - (D_3+D_2)+ \cdots (-1)^n (D_{n+2}+D_{n+1})\\
&=& (-1)^nD_{n+2}
\end{eqnarray*}
which completes the proof.
\end{proof}

Once having Theorem \ref{firsttheorem}, substitution of (\ref{derangementsrencontres}) in identity (\ref{thefirstdual}) gives the generalization (\ref{idthm1}). 
\begin{eqnarray} \label{idthm1}
1 + \sum_{k=1}^n \frac{D_{k+r}(r)}{  k!  \binom{k + r}{r} }  = \frac{D_{n+r+2}(r)}{(n+1)! \binom{n+r+2}{r}  }.
\end{eqnarray}

%\begin{proof}
% By iterative usage of the relation (\ref{themainrecurrence}) we have
%\begin{eqnarray*}
%\frac{D_{n+2}}{(n+1)!}&=& \frac{(n+1)(D_{n+1} +D_n )}{(n+1)!} = \frac{D_{n+1}}{n!} + \frac{D_n}{n!}\\
% &=& \frac{D_n}{(n+1)!} + \frac{D_{n-1}}{(n-1)!} +\frac{D_n}{n!} \\
%&=& 1 + \frac{D_1}{1!} + \cdots + \frac{D_n}{n!}\\
%&=& \sum_{k=0}^{n} \frac{D_k}{k!}.
%\end{eqnarray*}
%Now, the proof is completed by taking a substitution (\ref{derangementsrencontres}).
%\end{proof}

%Thus, the terms on l.h.s. of (\ref{idthm1}) are ratios with rencontres numbers of the related sequence, whose sum is equal to the similar ratio with the $(n+r+2)$-th rencontres number of the same sequence. In particular, when $r=0$ we deal with derangement numbers as Theorem \ref{firsttheorem} presents.

The identity (\ref{thm2identity}) follows by substitution of (\ref{derangementsrencontres}) in (\ref{theseconddual}),
\begin{eqnarray} \label{thm2identity}
1+ \sum_{k=1}^n \frac{(-1)^k D_{k+r+3}(r)}{(k+2) \binom{k+r+3}{r}} =     \frac{ (-1)^n D_{n+r+2}(r)}{\binom{n+r+2}{r}}.
\end{eqnarray}

%For example, when applying Corollary \ref{thesecondcor} to case $n=4, r=2$, we obtain $1 - 3 + 11 -53 +309 = 265 (=D_8(2)/36 ) $. 
Note that the terms in identity (\ref{thm2identity}) are always integers, which can be seen as a consequence of recurrence relation (\ref{themainrecurrence}).

%It is worth mentioning that rencontres numbers have further remarkable properties, including the fact that that $D_n$ is the permanent of the matrix of order $n$ %with 0 on the diagonal and 1 elsewhere. Moreover, rencontres numbers can be interpreted in context of other algebraic and combinatorial structures including deco %polyominoes, desarrangements  and necklaces \cite{OEIS}.

%Jos napraviti:\\
%\begin{itemize}
%\item Lema 1 i 2 trebaju biti teoremi a Tm 1 i 2 Cor1 i 2 (proci po tekstu...) - NAPRAVLJENO
%\item nacrtati sliku sa strane 4 moje pripreme - onu za Cn, crtati direktno u Letexu, u picture env. ne drugdje. isprekidana krivulja obicnom olovkom napisana, nije bitna
%\item ne sjecam se je li Lema 1 (sto treba izmju Tm1) dokazana onako algebraski kao sto je LEma 2, negdje prije u tekstu; cini mi se da nije - ako nije onda dodati takav dokaz - NAPRAVLJENO
%\item Jos malo vidjeti ono za Bathnagara, mozda to nije potreno navoditi - NAPRAVLJENO
%\item jos jednom procitati i provjeriti - NAPRAVLJENO
%\item skratiti - nije potrebno ono na pocetku dva puta dokazivati nego samo reci
%\item dotjerivanje: spell check, engleski gramatika
%\item dotjerivanje: jesu li sve kratice casopisa u literaturi ispravne...
%\end{itemize}

\end{document}